\documentclass[a4paper,12pt]{article}

\usepackage{amsthm,mathrsfs,amssymb,amsmath,bbm}
\usepackage{fullpage,enumerate}
\usepackage{makeidx}
\usepackage{graphicx}
\usepackage{xr,color}

\newtheorem{theorem}{Theorem}
\newtheorem{lemma}[theorem]{Lemma}
\newtheorem{proposition}[theorem]{Proposition}
\newtheorem{corollary}[theorem]{Corollary}

\theoremstyle{remark}
\newtheorem{definition}[theorem]{Definition}
\newtheorem{remark}[theorem]{Remark}
\newtheorem{example}[theorem]{Example}

\numberwithin{theorem}{section}
\numberwithin{equation}{section}

\newcommand{\RMe}{\mathrm{e}}

\newcommand{\R}{\mathbb{R}}
\renewcommand{\SS}{\mathbb{S}}
\newcommand{\Splus}[1][n-1]{\SS^{#1}_{+}}
\newcommand{\Sphere}[1][n-1]{\SS^{#1}}
\newcommand{\E}{{\mathbf E}}
\renewcommand{\P}{{\mathbf P}}

\newcommand{\Q}{{\mathbf Q}}

\newcommand{\deq}{\stackrel{d}{=}}

\newcommand{\sS}{\mathcal{S}}
\newcommand{\sK}{\mathscr{K}}

\newcommand{\sM}{\mathscr{M}}

\newcommand{\Lp}[1][p]{\mathsf{L}^{#1}}
\newcommand{\LL}{\mathsf{L}}

\DeclareMathOperator{\sign}{sign}

\newcommand{\one}{\mathbbm{1}}
\renewcommand{\emptyset}{\varnothing}

\newlength{\querylen}
\usepackage{fancybox}

\begin{document}

\title{Diagonal Minkowski classes, zonoid equivalence, and stable laws}

\author{Ilya~Molchanov, Felix~Nagel\\
  \small 
  Institute of Mathematical Statistics and Actuarial Science,\\
  \small University of Berne}

%\date{\today}
\date{}

\maketitle

\begin{abstract}
  We consider the family of convex bodies obtained from an origin
  symmetric convex body $K$ by multiplication with diagonal matrices,
  by forming Minkowski sums of the transformed sets, and by taking
  limits in the Hausdorff metric. Support functions of these convex
  bodies arise by an integral transform of measures on the family of
  diagonal matrices, equivalently, on Euclidean space, which we call
  $K$-transform. In the special case, if $K$ is a segment not lying on
  any coordinate hyperplane, one obtains the family of zonoids and the
  cosine transform. In this case two facts are known: the vector space
  generated by support functions of zonoids is dense in the family of
  support functions of origin symmetric convex bodies; and the cosine
  transform is injective. We show that these two properties are
  equivalent for general $K$.

  For $K$ being a generalised zonoid, we determine conditions that
  ensure the injectivity of the $K$-transform. Relations to mixed
  volumes and to a geometric description of one-sided stable laws are
  discussed. The later probabilistic application gives rise to a
  family of convex bodies obtained as limits of sums of diagonally
  scaled $\ell_p$-balls.


\end{abstract}

\section{Introduction}

\subsection{Geometric background}
\label{sec:geometric-background}

Typical transformations applied to convex bodies (non-empty convex
compact sets) in Euclidean space are scaling, translation and rotation,
or the whole group of invertible linear
transformations. Considerably less is known about the case when convex
bodies are transformed by actions of diagonal matrices, subsequently
called \emph{diagonal transformations}.

If $K$ is an origin symmetric \emph{segment} which does not lie on a
coordinate hyperplane, then the family of convex bodies generated by
diagonal transformations applied to $K$ is the same as the family
generated by rotations and scaling. For most other convex bodies $K$,
the two generated classes are not the same. While diagonal
transformations might seem too much dependent on the choice of the
coordinate system, their use is further motivated by a probabilistic
interpretation described below.

For two vectors $x,y\in\R^n$, 
\begin{displaymath}
  xy =(x_1y_1,\dots,x_ny_n)
\end{displaymath}
denotes their Hadamard (componentwise) product; for a convex body $K$
and any $u\in\R^n$,
\begin{displaymath}
  uK=\{ux:\; x\in K\}
\end{displaymath}
denotes $K$ transformed by the diagonal matrix with the diagonal
elements given by $u$. The Minkowski sum of two compact sets $K$ and
$L$ is $K+L=\{x+y:\; x\in K,y\in L\}$. The \emph{support function} of
a compact set $K$ is defined by
\begin{displaymath}
  h(K,u)=\sup\{\langle u,x\rangle:\; x\in K\}, \quad u\in\R^n,
\end{displaymath}
where $\langle \cdot,\cdot\rangle$ denotes the inner product in
$\R^n$.  In this paper we consider convex bodies that are symmetric,
which is always understood with respect to the origin.

Taking sums of diagonally transformed $K$ we obtain convex bodies
\begin{equation}
  \label{eq:21}
  c_1v_1K+\cdots +c_mv_mK
\end{equation}
for non-negative $c_1,\dots,c_m$ and $v_1,\dots,v_m$ from the unit
sphere. Passing to the limit, we arrive at the family of convex bodies
with support functions given by integrals of $h(vK,u)$ with respect to
a finite measure $\mu$ on the unit sphere $\Sphere$. These convex
bodies are called \emph{diagonal bodies generated by $K$}. Further,
for signed measures $\mu$ on $\Sphere$, consider the map
\begin{equation}
  \label{eq:13}
  \mu\mapsto (T_K\mu)(u)=\int_{\Sphere} h(vK,u)\mu(dv),
\end{equation}
which we call the \emph{$K$-transform of $\mu$}.
The central topic of this paper is the richness of the family of sets
given by \eqref{eq:21}, shown to be closely related to the injectivity
property of the $K$-transform.

Each symmetric segment with end-points $-u$ and $u$ is obtained as the
diagonal transformation $uI$ of the segment
\begin{equation}
  \label{eq:8}
  I=[-(1,\dots,1),(1,\dots,1)]
\end{equation}
with the end-points $(-1,\dots,-1)$ and $(1,\dots,1)$. Thus, applied
to $I$, the diagonal transformations yield the same family of sets as
scalings and rotations. Recall that $h(vI,u)=|\langle u,v\rangle|$. If
$\mu$ is a measure, $T_I\mu$ is the support function of a
\emph{zonoid}, see \cite[Th.~3.5.3]{schn2} and \cite{vi91e}.
\emph{Generalised zonoids} are defined by their support
functions $T_I\nu$ for signed measures $\nu$, more precisely, $K$ is a
generalised zonoid if 
\begin{equation}
  \label{eq:7}
  h(K,u) = \int_{\Sphere} |\langle u, v \rangle|\nu(dv),
%  =\E|\langle u,\zeta\rangle|, 
  \quad u \in \R^n,
\end{equation}
where $\nu$ is a finite signed measure on $\Sphere$ called the
\emph{representing measure of $K$}. 

Consider the family of convex bodies with support functions $T_K\mu$
for (non-negative) finite measures $\mu$.  This family is the smallest
Minkowski class that contains $K$ and is invariant under diagonal
transformations; we call it the \emph{diagonal Minkowski class} generated by
$K$.  If diagonal transformations are replaced by general invertible
linear transformations and $K$ belongs to a certain family
$\mathscr{A}$ of convex bodies, the generated
$\mathrm{GL}_n(\R)$-invariant Minkowski class has been considered in
\cite{ales03}; its elements are called
$\mathscr{A}$-bodies. Restriction to the group of rotations yields a
rotation invariant Minkowski class, see \cite{sch:sch07}.  The central
question in \cite{ales03,sch:sch07} concerns the richness of the
$\mathrm{GL}_n(\R)$-invariant or rotation invariant Minkowski classes.

\subsection{Probabilistic background}
\label{sec:prob-backgr}

From the probabilistic standpoint, the paper aims to determine which
information on integrable random vectors is available from the
expected values of certain sets of one-homogeneous even functions
applied to them. A function $f:\R^n\mapsto\R_+$ is said to be
\emph{one-homogeneous} if $f(cx)=cf(x)$ for all $c\geq 0$ and
$x\in\R^n$. A major result in this direction is the following theorem
from \cite{mol:sch:stuc14}.

\begin{theorem}
  \label{thr:m:s:s}
  Let $\xi$ and $\eta$ be two integrable random vectors in
  $\R^n$. Then
  \begin{equation}
    \label{eq:2}
    \E f(\xi)= \E f(\eta)
  \end{equation}
  for all measurable one-homogeneous even functions
  $f:\R^n\mapsto\R_+$ if and only if 
  \begin{equation}
    \label{eq:1}
    \E|\langle u,\xi\rangle|=\E|\langle u,\eta\rangle|, \quad u\in\R^n.
  \end{equation}
\end{theorem}

It is obvious that \eqref{eq:2} implies \eqref{eq:1}; the inverse
implication relies on the injectivity of the \emph{cosine transform}
\begin{equation}
  \label{eq:3}
  \mu\mapsto \int_{\Sphere} |\langle u,v\rangle|\mu(dv),
\end{equation}
which maps an even signed finite measure $\mu$ on the unit sphere $\Sphere$
to a homogenous function of $u\in\R^n$, see
\cite[Sec.~3.5]{schn2}. 
It is apparent that the cosine transform \eqref{eq:3} is a special
case of \eqref{eq:13} for $K=I$ from \eqref{eq:8}.

The right-hand side of \eqref{eq:3} is the support function of a
convex body in $\R^n$, called the \emph{zonoid} of $\mu$, see
\cite[Sec.~3.5]{schn2}. Similarly, $\E|\langle \xi,u\rangle|$ is the
support function of the zonoid of $\xi$, see \cite{vi91e} and
\cite{mos02}. Note that 
\begin{displaymath}
  \E|\langle \xi,u\rangle|=\E h(uI,\xi)=\E h(\xi I,u),
  \quad u\in\R^n.
\end{displaymath}
Two integrable random vectors $\xi$ and $\eta$ satisfying \eqref{eq:1}
are said to be \emph{zonoid equivalent}, see, e.g., \cite{mos02}.  The
extent to which zonoid equivalent random vectors can have different
distributions is explored in \cite{mol:sch:stuc14}.  

In the following we replace the function $|\langle u,x\rangle|$ above
with $h(uK,x)$ with the aim to describe convex bodies $K$ such that
the equality
\begin{equation}
  \label{eq:12}
  \E h(uK,\xi)=\E h(uK,\eta),\quad u\in\R^n,
\end{equation}
yields the zonoid equivalence of $\xi$ and $\eta$, hence, also
\eqref{eq:2} for all one-homogeneous even~$f$.  The zonoid equivalence
of $\xi$ and $\eta$ immediately yields \eqref{eq:12}.  The inverse
implication is not true for general $K$, as the following example
shows.

\begin{example}
  \label{ex:1}
  Let $K=[-1,1]^2$ be the $\ell_\infty$-ball in $\R^2$. For
  integrable $\xi=(\xi_1,\xi_2)$, we have 
  \begin{displaymath}
    \E h(uK,\xi)=\E|u_1\xi_1|+\E|u_2\xi_2|,
  \end{displaymath}
  which is determined by the first absolute moments of the coordinates
  of $\xi$, and so carries considerably less information than the
  zonoid of $\xi$.
\end{example}

\begin{example}
  \label{ex:2}
  Let $K=B_2$ be the unit Euclidean ball in $\R^n$. Then 
  \begin{displaymath}
    \E h(uB_2,\xi)=\E(u_1^2\xi_1^2+\cdots+u_n^2\xi_n^2)^{1/2}
  \end{displaymath}
  is the expected Euclidean norm of $u\xi$. Such expectations 
  carry the same information as the zonoid of $\xi$ if all components
  of $\xi$ are non-negative, see Theorems~\ref{thr:main} and
  \ref{th:lpzonoid}.
\end{example}

Note that $\E h(\xi K,u)$ is the support function of a convex body
$\E(\xi K)$ called the expectation of $\xi K$, see \cite{mo1}. In
particular, the zonoid of $\xi$ is the expectation of the random
segment $\xi I$, see \cite{vi91e}.

It is possible to define an $\Lp$-version of zonoid equivalence using
the $p$th moments of $|\langle u,\xi\rangle|$ or their
max-combinations. In this respect, the following result is important.

\begin{theorem}[see \protect{\cite[Th.~1.1]{wan:stoev10}}]
  \label{thr:ws}
  Let $\xi$ and $\eta$ be $\alpha$-integrable random vectors in $\R_+^n$
  with $\alpha\in(0,2)$. Then
  \begin{displaymath}
    \E |\langle u,\xi\rangle|^\alpha=\E |\langle u,\eta\rangle|^\alpha,
    \quad u\in \R^n,
  \end{displaymath}
  if and only if 
  \begin{displaymath}
    \E \big(\max_{i=1,\dots,n} u_i\xi_i\big)^\alpha 
    =\E \big(\max_{i=1,\dots,n} u_i\eta_i\big)^\alpha,
    \quad u\in\R_+^n. 
  \end{displaymath}
\end{theorem}

\subsection{Structure of the paper}
\label{sec:structure-paper}

The main aim of this paper is to study the diagonal Minkowski class
generated by a symmetric convex body $K$. In
Section~\ref{sec:affine-k-bodies}, an argument, relating
injectivity and surjectivity of a linear map, establishes the
equivalence of the injectivity of the $K$-transform \eqref{eq:13} and
the fact that the (signed) linear combinations of support functions of
diagonal bodies are dense in the family of support functions of
all symmetric convex bodies.

In Section~\ref{sec:affine-univ-gener} we explore in which cases $T_K$
is injective assuming that $K$ is a generalised zonoid. 
The key argument relies on deriving that the identity $\E
f(\xi\zeta)=\E f(\eta\zeta)$ for integrable symmetric random vectors
$\xi,\eta$, an independent random vector $\zeta$, and all
one-homogeneous even functions $f$ yields that $\E f(|\xi|)=\E
f(|\eta|)$, where $|\xi|$ and $|\eta|$ are vectors composed of the
absolute values of the components of $\xi$ and $\eta$,
respectively. In dimensions two and more it is generally not possible
to take $\zeta$ out of $f(\xi\zeta)$, and the proof goes via letting
$f$ be the product of powers of all arguments.

In order to show that \eqref{eq:12} implies \eqref{eq:1} (hence,
\eqref{eq:2}) for all integrable symmetric random vectors one has to
assume that $K$ does not possess certain symmetries. In particular, it
is necessary that $K$ is not symmetric with respect to any coordinate
hyperplane. The
major result of the paper, Theorem~\ref{thr:asym-uniqueness}, shows
that \eqref{eq:12} implies \eqref{eq:1} if and only if $K$ has
singe-point support sets in directions of all coordinate axes and
satisfies a specific asymmetry condition which is stronger than $K$
not being unconditional.

Section~\ref{sec:transf-surf-area} relates the $K$-transform of
surface area measures to the mixed volumes $V(L,\dots,L,uK)$ involving
the diagonally transformed $K$ and discusses the corresponding
uniqueness problem for the convex body $L$. 

Finally, in Section~\ref{sec:scaled-ell_p-balls} it is shown that
diagonally transformed $\ell_p$-balls are naturally related to
distributions of one-sided strictly stable random vectors. That
section complements the study of the geometric interpretation of
stable laws, initiated in \cite{mo09} for the symmetric setting. This
connection is exploited to show that $\ell_p$-balls (which are not
necessarily generalised zonoids) are unconditionally universal. Some
further properties of convex bodies obtained as diagonal transforms of
$\ell_p$-balls are derived.

The Appendix contains a result that extends the setup of
Section~\ref{sec:affine-k-bodies} to more general families of
transformations.

\section{Injectivity of the $K$-transform}
\label{sec:affine-k-bodies}

Convex bodies that can be approximated in the Hausdorff metric by 
positive linear combinations of a convex body $K$ and its rotations
form a Minkowski class $\mathscr{M}$; its members are called
\emph{$\mathscr{M}$-bodies}, see \cite{sch:sch07}. Furthermore, $M$ is
a \emph{generalised $\mathscr{M}$-body} if
\begin{equation}
  \label{eq:5}
  M+L_1=L_2
\end{equation}
with $L_1,L_2$ being $\mathscr{M}$-bodies.

Denote by $\sK_0$ the family of origin symmetric %(also called centred)
convex bodies. A convex body $K\in\sK_0$ is called \emph{centrally
  universal} if the expansion of its support function into spherical
harmonics contains non-zero harmonics of all even orders, see
\cite{schn74}. It is shown in \cite[Th.~2]{sch:sch07} that this holds
if and only if the family of generalised $\mathscr{M}$-bodies is dense
in $\sK_0$. In particular, if $K$ is a symmetric segment, the
corresponding $K$-bodies are zonoids and generalised zonoids are dense
in $\sK_0$.

In this paper, we consider a similar situation, but instead of
rotations we apply to $K$ diagonal transformations. % In other words, we
Recall that the rotations/scalings and diagonal transformations are
equivalent if $K$ is the segment $I$ given by \eqref{eq:8}.

\begin{theorem}
  \label{thr:equiv}
  Let $K \in \sK_0$.  The linear combinations of the support functions
  of diagonal bodies generated by $K$ are dense in the family of continuous even
  functions on the unit sphere with the uniform metric (hence, in the
  family of support functions of symmetric convex bodies) if and only
  if the $K$-transform is injective on finite even
  signed measures on $\Sphere$.
\end{theorem}
\begin{proof}
  \textsl{Necessity.}  The support function of the segment $I$ can be
  approximated by linear combinations of support functions of diagonal
  bodies, and it suffices to refer to the injectivity property of
  the cosine transform.
 
  \textsl{Sufficiency.} Denote by $\sM_e$ the family of finite signed
  even measures on the unit sphere and by $C_e$ the family of even 
  continuous functions on the unit sphere. We will prove that the
  image of $T_K$ restricted to measures with densities from $C_e$ is
  dense in $C_e$. Note that $T_K$ is continuous in the uniform metric
  on $C_e$.  Its adjoint operator $T'_K$ acts on $\sM_e$, so that
  \begin{displaymath}
    \int_{\Sphere} f d(T'_K\nu) =\int_{\Sphere} (T_K f) d\nu,
    \quad f\in C_e. 
  \end{displaymath}
  Assume that $T'_K\nu=0$, that is, $\nu$ belongs to the kernel of
  $T'_K$. Then 
  \begin{displaymath}
    \int_{\Sphere}\int_{\Sphere} h(uK,v)f(u)\,du\,\nu(dv)=0
  \end{displaymath}
  for all $f\in C_e$, whence
  \begin{displaymath}
    \int_{\Sphere} h(vK,u)\nu(dv)=0,\quad u\in\Sphere.
  \end{displaymath}
  By the injectivity assumption, $\nu=0$. The triviality of the kernel of
  $T'_K$ yields that the range of $T_K$ is dense in $C_e$, see
  \cite[Th.~III.4.5]{wern00}.
\end{proof}

In analogy with \cite{sch:sch07}, $K$ satisfying one of the equivalent
conditions in Theorem~\ref{thr:equiv} is called \emph{diagonally
  universal} or \emph{D-universal}. 

\begin{proposition}
  \label{prop:sphere}
  A set $K\in\sK_0$ is D-universal if and only if \eqref{eq:12} for
  any two symmetric integrable random vectors $\xi$ and $\eta$ implies
  their zonoid equivalence.
\end{proposition}
\begin{proof}
  \textsl{Necessity.}  A symmetric integrable random vector $\xi$
  yields an even finite measure on $\Sphere$ given by
  \begin{displaymath}
    \mu_\xi(A)=\E(\|\xi\| \, \one_{\xi/\|\xi\|\in A})
  \end{displaymath}
  for all Borel $A\subset\Sphere$, where $\|x\|$ denotes the Euclidean
  norm of $x \in \R^n$. Then
  \begin{displaymath}
    \E h(uK,\xi)=\int_{\Sphere} h(uK,v)\mu_\xi(dv),\quad u\in\R^n.
  \end{displaymath}
  Thus, if $T_K$ is injective, then \eqref{eq:12} yields
  $\mu_\xi=\mu_\eta$, whence \eqref{eq:1} holds and $\xi$ and $\eta$
  are zonoid equivalent.

  \textsl{Sufficiency.} A finite measure $\mu$ on $\Sphere$ yields
  an integrable random vector $\xi=\mu(\Sphere)\xi'$, where $\xi'$ is
  distributed according to the normalised $\mu$. 
  Then 
  \begin{displaymath}
    \int_{\Sphere} h(uK,v)\mu(dv)=\E h(uK,\xi),\quad u\in\R^n.
  \end{displaymath}
  Let $\eta$ be generated by $\nu$ in the same way. If
  $T_K\mu=T_K\nu$, then $\xi$ and $\eta$ are zonoid equivalent, whence
  $\mu=\nu$ by the injectivity of the standard cosine transform. The
  injectivity of $T_K$ on signed even measures follows from the Jordan
  decomposition into their positive and negative parts. 
\end{proof}

A convex body is called \emph{unconditional} if it is symmetric with
respect to all coordinate hyperplanes $\{x=(x_1,\dots,x_n)\in\R^n:\;
x_i=0\}$, $i=1,\dots,n$; the family of unconditional bodies is denoted
by $\sK_s$.

\begin{theorem}
  \label{thr:equiv-pos}
  Let $K \in \sK_s$. The linear combinations of support functions of
  diagonal bodies generated by $K$ are dense in the family of support functions of
  unconditional convex bodies if and only if \eqref{eq:12} for all
  $u\in\R^n$ and any two integrable random vectors $\xi$ and $\eta$ in
  $\R^n$ implies the zonoid equivalence of
  $|\xi|=(|\xi_1|,\dots,|\xi_n|)$ and
  $|\eta|=(|\eta_1|,\dots,|\eta_n|)$.
\end{theorem}
\begin{proof}
  Since $K$ is unconditional, \eqref{eq:12} holds if and only if 
  \begin{displaymath}
    \E h(uK,|\xi|)=\E h(uK,|\eta|),
  \end{displaymath}
  whence it suffices to assume that $\xi$ and $\eta$ take values in
  $\R_+^n$ and let $u\in\R_+^n$. 

  \textsl{Necessity.}  By assumption, the support function of the
  $\ell_1$-ball $B_1$ can be approximated by linear combinations of
  support functions of diagonal bodies, so that
  \begin{displaymath}
    \E\max(u_1\xi_1,\dots,u_n\xi_n)
    = \E\max(u_1\eta_1,\dots,u_n\eta_n)
  \end{displaymath}
  for all $u=(u_1,\dots,u_n)\in\R_+^n$. By
  Theorem~\ref{thr:ws}, $\xi$ and $\eta$ are zonoid equivalent.
  
  \textsl{Sufficiency.} The proof replicates the proof of sufficiency
  in Theorem~\ref{thr:equiv} by restricting measures and functions
  onto the unit sphere intersected with $\R_+^n$. 
\end{proof}

The property of $K$ (not necessarily belonging to $\sK_s$) formulated
in Theorem~\ref{thr:equiv-pos} can be interpreted as the injectivity
of the $K$-transform over unconditional measures. In this case $K$ is
called \emph{unconditionally D-universal}.

\begin{example}
  \label{ex:max}
  It follows from Theorem~\ref{thr:ws} that $B_1$ is unconditionally
  D-universal. 
  The derivation of similar results for $B_q$ with $q>1$
  requires the methods of Section~\ref{sec:scaled-ell_p-balls}.
\end{example}

The following result shows that a scaled $\ell_\infty$-ball is allowed
as a summand of a D-universal $K$ if the rest is D-universal.

\begin{proposition}
  \label{prop:segments}
  Let $B_\infty=\{x\in\R^n:\; |x_i|\leq 1, i=1,\dots,n\}$ be the
  $\ell_\infty$-ball in $\R^n$, and let 
  $K=K_0+wB_\infty$ for $w\in\R_+^n$ and $K_0\in\sK_0$. Then
  \eqref{eq:12} implies 
  \begin{equation}
    \label{eq:14}
    \E h(uK_0,\xi)=\E h(uK_0,\eta), \quad u\in\R^n.
  \end{equation}
\end{proposition}
\begin{proof}
  Since $wB_\infty$ equals the sum of $w_i[-e_i,e_i]$, $i=1,\dots,n$, it
  is possible to proceed by induction. Let $K=K_0+[-e_1,e_1]$. If
  $u=e_1$, then 
  \begin{displaymath}
    \E h(uK,\xi)=(h(K_0,e_1)+1)\E |\xi_1|. 
  \end{displaymath}
  Thus, \eqref{eq:12} implies $\E|\xi_1|=\E|\eta_1|$. For $u\in\R^n$, 
  \begin{displaymath}
    \E h(uK,\xi)=\E h(uK_0,\xi)+|u_1|\E |\xi_1|,  
  \end{displaymath}
  whence \eqref{eq:14} holds. 
\end{proof}

Thus, if $K_0$ is (unconditionally) D-universal, then also the convex
body $K_0+vB_\infty$ is (unconditionally) D-universal for all
$v\in\R^n$.

\begin{remark}
  In order to eliminate the dependence of the diagonal transformations
  on the chosen coordinate system, it is possible to incorporate a
  single rotation of $K$, that is, consider the Minkowski class
  generated by $c_1u_1OK+\cdots c_mu_mOK$, where $O$ is an orthogonal
  matrix.\footnote{The authors are grateful to Daniel Hug for
    suggesting this construction.} The universality holds if and only
  if a rotation of $K$ is D-universal. 
\end{remark}

\begin{remark}
  If $p\in[1,\infty)$, the $L_p$-variant of the $T_K$-transform is
  defined by letting
  \begin{displaymath}
    T_K^p\mu =\left(\int_{\Sphere} h(uK,v)^p\mu(dv)\right)^{1/p},
  \end{displaymath}
  equivalently, as $(\E h(uK,\xi)^p)^{1/p}$ for a $p$-integrable
  random vector $\xi$.  Furthermore, it is possible to consider radial
  sums of diagonally transformed convex bodies or add them in the
  Blaschke sense.
\end{remark}

\section{Diagonal universality of generalised zonoids}
\label{sec:affine-univ-gener}

\subsection{Unconditional universality}
\label{sec:uncond-univ}

For $u\neq 0$, let
\begin{displaymath}
  F(K,u)=\{x\in K:\; \langle x,u\rangle=h(K,u)\}
\end{displaymath}
denote the support set of $K$ in direction $u$. 

\begin{lemma}
  \label{le:sphere}
  For a generalised zonoid $K$, the support sets $F(K,e_i)$ in
  directions of the standard basis vectors $e_1,\dots,e_n$ are all
  singletons if and only if the representing measure of $K$ vanishes on
  $S_0=\{v\in\Sphere:\; v_1\cdots v_n=0\}$.
\end{lemma}
\begin{proof}
  By \cite[Lemma~3.5.6]{schn2},  the
  support set $F(K,u)$ is a singleton for $u\neq 0$ if and only if
  the representing measure vanishes on
  \begin{math}
    S_u=\{x\in\Sphere:\; \langle x,u\rangle =0\}.
  \end{math}
  Now the claim follows from the fact that $S_0 = \cup_i\, S_{e_i}$.
\end{proof}

Recall that $K$ is unconditionally D-universal if \eqref{eq:12}
implies that $(|\xi_1|,\dots,|\xi_n|)$ and $(|\eta_1|,\dots,|\eta_n|)$
are zonoid equivalent.

\begin{theorem}
  \label{thr:main}
  Let $K$ be a generalised zonoid such that all support sets
  $F(K,e_i)$, $i=1,\dots,n$, are singletons.  Then $K$ is
  unconditionally D-universal.
\end{theorem}

The proof of Theorem~\ref{thr:main} is based on the following two
lemmas. For $x\in\R^n$ and $\alpha\in\R^n$, write
\begin{displaymath}
  [x]^\alpha=\prod_{i=1}^n |x_i|^{\alpha_i},
\end{displaymath}
and 
\begin{align*}
  |x|&=(|x_1|,\dots,|x_n|),\\
  \sign(x)&=(\sign(x_1),\dots,\sign(x_n))
\end{align*}
for the componentwise absolute values and signs of $x\in\R^n$. We use
the convention $0^0=1$. For $E\subseteq\{1,\dots,n\}$, denote
\begin{displaymath}
  A_E=\{x \in\R^n:\; x_i\neq 0, i\in E,x_j=0, j\notin E\},
\end{displaymath}
and let 
\begin{equation}
  \label{eq:19}
  \Delta_E=\Big\{\alpha\in\R_+^n:\; \sum \alpha_i=1, \alpha_j=0,j\notin E\Big\}
\end{equation}
be the unit simplex in the linear subspace generated by $A_E$. 
Let $A_\emptyset=\{0\}$ and $\Delta_\emptyset=\{0\}$.

\begin{lemma}
  \label{lemma:kabluchko}
  Let $\xi$ and $\eta$ be two random vectors in $(0,\infty)^n$. If 
  \begin{displaymath}
    \E [\xi]^\alpha=\E [\eta]^\alpha < \infty
  \end{displaymath}
  for all $\alpha$ from an open set in $\R^n$, then $\xi$ and $\eta$ are identically
  distributed.
\end{lemma}
\begin{proof}
  By passing to componentwise logarithms $\tilde\xi=\log\xi$ and
  $\tilde\eta=\log\eta$, we see that the Laplace transforms of
  $\tilde\xi$ and $\tilde\eta$ agree on an open set, and the result
  follows from \cite[Lemma~7]{kab:sch:haan09}.
\end{proof}

\begin{lemma}
  \label{lemma:powers-all}
  Let $\xi$ and $\eta$ be integrable random vectors in $\R_+^n$. Then
  $\xi$ and $\eta$ are zonoid equivalent if and only if
  \begin{equation}
    \label{eq:11}
    \E \big([\xi]^\alpha\one_{\xi\in A_E}\big)
    =\E \big([\eta]^\alpha\one_{\eta\in A_E}\big)
  \end{equation}
  for all non-empty $E\subseteq\{1,\dots,n\}$ and $\alpha$ from a
  relatively open subset of the unit simplex $\Delta_E$.
\end{lemma}
\begin{proof}
  \textsl{Necessity} follows from the fact that, for all
  $\alpha\in\Delta_E$, the expectations in \eqref{eq:11} are taken of
  a one-homogenous even function of $\xi$ and $\eta$, so that
  Theorem~\ref{thr:m:s:s} applies.

  \textsl{Sufficiency.} Fix $E\subseteq\{1,\dots,n\}$ of cardinality
  at least $2$, and assume that \eqref{eq:11} holds for $\alpha$ from
  a relatively open neighbourhood of some $\beta\in\Delta_E$. Then
  \begin{equation}
    \label{eq:10}
    c=\E \big([\xi]^\beta\one_{\xi\in A_E}\big)
    =\E \big([\eta]^\beta\one_{\eta\in A_E}\big).
  \end{equation}
  Assume that $c>0$ and define probability measures $\Q$ and $\Q^*$
  with densities
  \begin{align*}
    \frac{d\Q}{d\P}=\frac{1}{c}[\xi]^\beta\one_{\xi\in A_E},\qquad
    \frac{d\Q^*}{d\P}=\frac{1}{c} [\eta]^\beta\one_{\eta\in A_E},
  \end{align*}
  and let $\E_\Q$ and $\E_{\Q^*}$ denote the expectations with respect
  to $\Q$ and $\Q^*$, respectively. 

  Then \eqref{eq:11} implies
  \begin{displaymath}
    \E_\Q \prod_{i\in E} \xi_i^{\gamma_i}
    =\E_{\Q^*} \prod_{i\in E} \eta_i^{\gamma_i}
  \end{displaymath}
  for all $\gamma$ from a neighbourhood of the origin and such that $\sum
  \gamma_i=0$ and $\gamma_j=0$ for $j\notin E$. Fix any $k\in E$ and
  notice that $\gamma_k=-\sum_{i\neq k}\gamma_i$. Therefore, 
  \begin{displaymath}
    \E_\Q \prod_{i\in E, i\neq k}
      \left(\frac{\xi_i}{\xi_k}\right)^{\gamma_i}
    =\E_{\Q^*} \prod_{i\in E, i\neq k}
      \left(\frac{\eta_i}{\eta_k}\right)^{\gamma_i}
  \end{displaymath}
  for $\gamma_i$, $i\neq k$, from an open set in the subspace of
  $\R^n$ generated by the basis vectors $e_i$, $i\in E\setminus\{k\}$. 
  
  By Lemma~\ref{lemma:kabluchko} applied to this subspace of $\R^n$,
  the distributions of $\xi_i/\xi_k$, $i\in E\setminus\{k\}$, under
  $\Q$ and $\eta_i/\eta_k$, $i\in E\setminus\{k\}$, under $\Q^*$
  coincide. These vectors can be extended by placing $1$ at the
  component number $k$ and $0$ at all components outside $E$. Therefore, 
  \begin{align*}
    \E \big(|\langle u,\xi\rangle|\one_{\xi\in A_E}\big)
    &=c\,\E_\Q \big(|\langle u,\xi_k^{-1}\xi\rangle|
      \,[\xi_k^{-1}\xi]^{-\beta}\one_{\xi \in A_E}\big)\\
    &=c\,\E_{\Q^*} \big(|\langle u,\eta_k^{-1}\eta\rangle|
      \,[\eta_k^{-1}\eta]^{-\beta}\one_{\eta \in A_E}\big)\\
    &=\E \big(|\langle u,\eta\rangle|\one_{\eta \in A_E}\big).
  \end{align*}
  This also holds if $c=0$ in \eqref{eq:10}, and for $E=\emptyset$.
  If $E$ is a singleton, then this holds by \eqref{eq:11}.  Taking the
  sum over all $E\subseteq\{1,\dots,n\}$ shows that $\xi$ and $\eta$
  are zonoid equivalent.
\end{proof}

\begin{proof}[Proof of Theorem~\ref{thr:main}]
  It follows from \eqref{eq:12} and \eqref{eq:7} that
  \begin{displaymath}
    \int_{\Sphere} \E|\langle u,\xi v\rangle|\nu(dv)
    =\int_{\Sphere} \E|\langle u,\eta v\rangle|\nu(dv), \quad u\in\R^n. 
  \end{displaymath}
  By splitting $\nu$ into the positive and negative parts and
  referring to the equivalence of \eqref{eq:1} and \eqref{eq:2} (see
  \cite[Th.~2]{mol:sch:stuc14}), it is possible to conclude that 
  \begin{displaymath}
    \int_{\Sphere} \E f(\xi v) \nu(dv)
    =\int_{\Sphere} \E f(\eta v) \nu(dv)
  \end{displaymath}
  for all one-homogeneous even functions $f$. 

  Taking $f(x)=[x]^\alpha  \one_{x\in A_E}$ for non-empty
  $E\subseteq\{1,\dots,n\}$ yields that 
  \begin{displaymath}
    \int_{\Sphere} \E \big([\xi v]^\alpha \one_{\xi v\in A_E}\big)\nu(dv)
    = \int_{\Sphere}\E \big([\eta v]^\alpha \one_{\eta v\in
      A_E}\big)\nu(dv), 
    \quad \alpha\in\Delta_E. 
  \end{displaymath}
  Taking into account that $\nu$ vanishes on the set $\{x\in\Sphere:\;
  x_1\cdots x_n=0\}$ by Lemma~\ref{le:sphere}, we may assume that
  none of the components of $v$ from the integration domain vanishes,
  whence 
  \begin{displaymath}
    \int_{\Sphere} [v]^\alpha\nu(dv) \, \E \big([\xi]^\alpha \one_{\xi\in A_E}\big)
    = \int_{\Sphere}[v]^\alpha\nu(dv) \, \E \big([\eta]^\alpha \one_{\eta\in A_E}\big), 
    \quad \alpha\in\Delta_E. 
  \end{displaymath}
  The integral of $[v]^\alpha$ does not vanish for $\alpha$
  being any basis vector. Indeed, if it vanishes for $\alpha=e_i$, then 
  \begin{displaymath}
    0=\int_{\Sphere} [v]^{e_i}\nu(dv)
    =\int_{\Sphere} |\langle v,e_i\rangle|\nu(dv)=h(K,e_i),
  \end{displaymath}
  whence $K$ is a subset of a coordinate hyperplane.

  Hence,
  \begin{displaymath}
    \E \big([\xi]^\alpha \one_{\xi\in A_E}\big)
    =\E \big([\eta]^\alpha \one_{\eta\in A_E}\big)
  \end{displaymath}
  for all $\alpha$ from a neighbourhood in $\Delta_E$ of any basis
  vector $e_k$ with $k\in E$.  By Lemma~\ref{lemma:powers-all}, $|\xi|$
  and $|\eta|$ are zonoid equivalent.
\end{proof}  

\begin{corollary}
  \label{cor:zon-uncond}
 Each symmetric convex body in $\R^2$ is either a diagonal transform of $B_\infty$ or it is unconditionally D-universal.
\end{corollary}
\begin{proof}
  Recall that in $\R^2$ each symmetric convex body $K$ is a zonoid, say with representing measure $\nu$ on~$S^1$. Restricting $\nu$ to $S^1 \setminus \{ \pm e_1, \pm e_2 \}$ and to $\{ \pm e_1, \pm e_2 \}$ yields a decomposition $K = K_0 + w B_\infty$ where $K_0$ is a zonoid with support sets $F(K_0,e_i)$, $i=1,2$, being singletons and $w \in \R^2$. If $K_0$ is non-trivial, then it is unconditionally D-universal by Theorem~\ref{thr:main}, and thus $K$ is unconditionally D-universal by Proposititon~\ref{prop:segments}.
\end{proof}

Notice that the generalisation of Corollary~\ref{cor:zon-uncond} to dimension $n \geq 3$ fails, even for zonoids, because the restriction of $\nu$ to $\Sphere$ and $\Sphere \setminus S_0$ does not lead to a decomposition of the form considered in Proposition~\ref{prop:segments}.

\subsection{Diagonal universality} %: discovering the signs}
\label{sec:d-universality}

Theorem~\ref{thr:main} yields that all generalised zonoids with
single-point support sets in directions of all basis vectors are
\emph{unconditionally} D-universal. To obtain the D-universality
property, one has to be able to recover the distribution of signs of
the components of $\xi$ from $\E h(uK,\xi)$.
For $J\subset\{1,\dots,n\}$ and $\alpha\in\R_+^n$, introduce the
function
\begin{equation}
  \label{eq:17}
  f_{\alpha,J}(x)%=\prod_{i=1}^n |x_i|^{\alpha_i} \prod_{i\in J}\sign(x_i)
  =[x]^\alpha \prod_{i\in J} \sign(x_i), \quad x\in\R^n. 
\end{equation}
Denoting by $t^{\langle\beta\rangle}=|t|^\beta\sign (t)$ the signed
$\beta$-power of $t\in\R$ for $\beta\geq0$, it is possible to write
\begin{displaymath}
  f_{\alpha,J}(x)=\prod_{i\in J} x_i^{\langle\alpha_i\rangle}
  \prod_{i\notin J} |x_i|^{\alpha_i}. 
\end{displaymath}

For a symmetric convex body $K$ in $\R^n$ and
$J\subseteq\{1,\dots,n\}$ the projection of $K$ on the subspace $\R^J$
spanned by $\{e_j,j \in J\}$ is denoted by $K_J$. Moreover, denote
$\sS^n = \{ -1, 1\}^n$ and $\sS^J = \{ -1, 1\}^J$, and write
$\sigma(t) = \prod_{j \in J} \sign(t_j)$ for $t \in \R^J$ or $t \in
\sS^J$. For $s\in\sS^n$, denote by $R_s$ the set of all $x\in\R^n$ such
that $\sign(x)=s$. 

A symmetric convex body $L$ in $\R^J$ is said to satisfy the
\emph{asymmetry condition} if
\begin{equation}
  \label{eq:27}
  \sum_{s\in\sS^J\!,\;\sigma(s)=1} \!\!\!\! sL \;\;
  \neq \sum_{s\in\sS^J\!,\;\sigma(s)=-1} \!\!\!\! sL.
  \tag{AS}
\end{equation}
For this, the cardinality of $J$ should be even; \eqref{eq:27} never
holds if the cardinality of $J$ is odd. It is easy to see
that \eqref{eq:27} is equivalent to $L$ being not unconditional if the cardinality of $J$ is~$2$; in higher dimensions, \eqref{eq:27} is stronger than $L$ not
being unconditional as can be seen from Example~\ref{ex:as4dim} below.

\begin{theorem}
  \label{thr:asym-uniqueness}
  Let $K$ be a generalised zonoid in $\R^n$ such that all support sets
  $F(K,e_i)$, $i=1,\dots,n$, are singletons. Then $K$ is D-universal
  if and only if $K_J$ satisfies the asymmetry condition \eqref{eq:27}
  for each non-empty $J\subseteq\{1,\dots,n\}$ of even cardinality.
\end{theorem}

To prove sufficiency we need a generalisation of Lemma~\ref{lemma:powers-all}.

\begin{lemma}
  \label{lemma:powers-sign}
  Two symmetric integrable random vectors $\xi$ and $\eta$ in $\R^n$
  are zonoid equivalent if and only if
  \begin{equation}
    \label{eq:20}
    \E \big(f_{\alpha,J}(\xi)\one_{\xi\in A_E}\big)
    =\E \big(f_{\alpha,J}(\eta)\one_{\eta\in A_E}\big)
  \end{equation}
  for all $J\subseteq E\subseteq\{1,\dots,n\}$ with $E\neq\emptyset$
  and all $\alpha$ from a relatively open set in $\Delta_E$.
\end{lemma}
\begin{proof}
\textsl{Necessity.} Theorem~\ref{thr:m:s:s} implies \eqref{eq:20} if the cardinality of
  $J$ is even, otherwise both sides of \eqref{eq:20}
  vanish.\\
\textsl{Sufficiency.} Let $s\in\{-1,0,1\}^n$ and $E$ be such
  that $s_i\neq 0$ for $i\in E$ and $s_i=0$ for $i\notin E$. Then
  \begin{displaymath}
    \sum_{J\subseteq E} \prod_{i\in J}s_i\sign(x_i)\one_{x\in A_E} 
    =2^{m} \one_{\sign(x)=s}, \quad x\in\R^n,
  \end{displaymath}
  where $m$ is the cardinality of $E$.  Using this in \eqref{eq:20}, we have
  \begin{displaymath}
    \E \big([\xi]^\alpha\one_{\xi\in A_E}\one_{\sign(\xi)=s}\big)
    =\E \big([\eta]^\alpha\one_{\eta\in A_E}\one_{\sign(\eta)=s}\big).
  \end{displaymath}
  By Lemma~\ref{lemma:powers-all}, $|\xi|\one_{\sign(\xi)=s}$ and
  $|\eta|\one_{\sign(\eta)=s}$ are zonoid equivalent, whence 
  \begin{displaymath}
    \E \big(|\langle u,|\xi|\rangle|\one_{\sign(\xi)=s}\big)
    =\E \big(|\langle u,|\eta|\rangle|\one_{\sign(\eta)=s}\big), \quad u\in\R^n.
  \end{displaymath}
  Since this equality holds for all $u$, it is possible to replace $u$
  by $us$, and it remains to take the sum over $s$.
\end{proof}

\begin{proof}[Proof of Theorem~\ref{thr:asym-uniqueness}]
  \textsl{Necessity.} Assume that $K_J$ does not satisfy \eqref{eq:27} for some non-empty $J$ with even cardinality. Then 
  \begin{equation}
    \label{eq:29}
    \sum_{s\in\sS^J\!,\;\sigma(s)=1} \!\!\!\! sL \;\;=
    \sum_{s\in\sS^J\!,\;\sigma(s)=-1} \!\!\!\! sL
  \end{equation}
  holds for all convex bodies $L$ in $\R^J$, whose support functions
  lie in the closed linear hull of support functions of the diagonal
  transforms of~$K_J$. Since the unit segment $I$ in $\R^J$ does not
  satisfy \eqref{eq:29}, $I$ cannot be approximated which is a
  contradiction to $K$ being D-universal.
  
  \textsl{Sufficiency.}  
  The representing measure of $K$ is denoted by~$\nu$. Let $\xi$ and
  $\eta$ be symmetric random vectors. Repeating the first step in the
  proof of Theorem~\ref{thr:main}, we arrive at
  \begin{displaymath}
    \int_{\Sphere} \E \big(f_{\alpha,J}(\xi v)\one_{\xi v\in A_E}\big) \nu(dv)
    =\int_{\Sphere}\E \big(f_{\alpha,J}(\eta v)\one_{\eta v\in A_E}\big)  \nu(dv)
  \end{displaymath}
  for all $\alpha\in\Delta_E$ and $J\subseteq E\subseteq\{1,\dots,n\}$
  where $J\neq\emptyset$ and $J$ has even cardinality. As $\nu$ does
  not charge $S_0$,
  \begin{displaymath}
    \int_{\Sphere} f_{\alpha,J}(v)\nu(dv)
    \Big(\E \big(f_{\alpha,J}(\xi)\one_{\xi\in A_E}\big)
    -\E \big(f_{\alpha,J}(\eta)\one_{\eta\in A_E}\big)\Big)=0.
  \end{displaymath}
  Now fix $J$. We show that
  \begin{equation}
    \label{eq:Jneq}
    \int_{\Sphere}  f_{\alpha,J}(v)\nu(dv) \neq 0
  \end{equation}
  for some $\alpha\in\Delta_J\subseteq\Delta_E$, and therefore also
  for all $\alpha$ in a relatively open set in $\Delta_E$ by
  continuity of $[x]^\alpha$ in $\alpha$ for $x$ with non-vanishing
  components. Assume that
  \begin{equation}
    \int_{\SS^J}  f_{\alpha,J}(w) \nu_J(dw)
    = \int_{\Sphere}  f_{\alpha,J}(v) \nu(dv) = 0,\quad \alpha \in \Delta_J,
  \end{equation}
  where $\SS^J$ is the unit sphere in $\R^J$ and $\nu_J$ is the
  representing measure of~$K_J$. We assume that $J = \{1,\dots,n\}$
  with an even $n$ and $\nu=\nu_J$, the general case being
  similar. Then
  \begin{align*}
    \int_{\Sphere}  f_{\alpha,J}(v) \nu(dv) 
    &=  \int_{\Sphere} \prod_{i=1}^n v_i^{\langle \alpha_i \rangle} \nu(dv) \\
    & = \sum_{s \in \sS^n} \int_{\Sphere\cap R_s} [v]^\alpha \sigma(v) \nu(dv) \\
    & =  \sum_{s \in \sS^n} \int_{\Sphere\cap R_s}
    [v]^\alpha \sigma(s) \nu(dv), 
  \end{align*}
  where the last equality holds because $\sigma(v) = \sigma(s)$ for
  $v\in R_s$. Since the expression in the last line is zero for $\alpha \in
  \Delta_J$, Lemma~\ref{lemma:powers-all} implies that
  \begin{displaymath}
    0 = \sum_{s\in\sS^n} \sigma(s) \int_{\Sphere\cap R_s} |\langle u,|v|\rangle| \nu(dv) =
    \sum_{s\in\sS^n} \sigma(s) \int_{\Sphere\cap R_s} |\langle su,v\rangle| \nu(dv),\quad u\in\R^n.
  \end{displaymath}
  Multiplying this with $\sigma(t)$, replacing $u$ with $ut$, taking
  the sum over $t\in\sS^n$, and writing $t$ instead of $ts$ yield that
  \begin{displaymath}
    0 = \sum_{t\in\sS^n} \sigma(t)\sum_{s\in\sS^n}\int_{\Sphere\cap R_s}|\langle tu,v\rangle| \nu(dv) = \sum_{t\in\sS^n} \sigma(t)\int_{\Sphere}|\langle tu,v\rangle| \nu(dv),\quad u\in\R^n.
  \end{displaymath}
  Therefore, for $u \in \R^n$
  \begin{displaymath}
    \sum_{s \in \sS^n} \sigma(s) h(K, su)
    =  \sum_{s \in \sS^n} \sigma(s)
    \int_{\Sphere} |\langle v, su \rangle| \nu(dv) = 0.
  \end{displaymath}
  This contradicts \eqref{eq:27}, so \eqref{eq:Jneq} is proven.

  If the cardinality of $J$ is odd, then the symmetry of $\xi$ and
  $\eta$ yields that \eqref{eq:20} holds with both sides vanishing. If
  $J=\emptyset$ and $E\neq\emptyset$, then \eqref{eq:20} holds
  since $|\xi|$ and $|\eta|$ are zonoid equivalent by Theorem~\ref{thr:main}. Finally, the
  result follows from Lemma~\ref{lemma:powers-sign}.
\end{proof}

\begin{corollary}
  \label{co:2-dim}
  A symmetric convex body $K$ in $\R^2$ is D-universal if and
  only if $K$ is not unconditional.
\end{corollary}
\begin{proof}
  Decompose $K$ as in the proof of Corollary~\ref{cor:zon-uncond}. If
  $K$ is not unconditional, then $K_0$ is not unconditional, which in
  $\R^2$ implies \eqref{eq:27}. Hence $K_0$ is D-universal by
  Theorem~\ref{thr:asym-uniqueness}, and therefore $K$ is D-universal
  by Proposititon~\ref{prop:segments}. Necessity is clear since only
  unconditional convex bodies are generated by unconditional $K$. 
\end{proof}

\begin{corollary}
  A generalised zonoid $K$ in $\R^3$ such that $F(K,e_i)$, $i=1,2,3$,
  are singletons is D-universal if and only if none of the
  two-dimensional projections of $K$ is unconditional.
\end{corollary}

\begin{example}
  \label{ex:as4dim}
  In four dimensions consider convex body $K =
  \sum_{s\in\sS^4} a_s s I$. If 
  \begin{align*}
    a_{(1,1,1,1)} & = 1,& a_{(-1,-1,1,1)} & = 2,
    & a_{(-1,1,-1,1)} & = 20,& a_{(-1,1,1,-1)} & = 24, \\
    a_{(-1,1,1,1)} & = 18,& a_{(1,-1,1,1)} & = 17,
    & a_{(1,1,-1,1)} & = 4,& a_{(1,1,1,-1)} & = 8,
  \end{align*}
  then all six two-dimensional projections of $K$ satisfy the
  asymmetry condition (in particular, $K$ is not unconditional),
  whereas the four-dimensional case of condition \eqref{eq:27}
  fails. 
\end{example}

\subsection{Transformation of surface area measures}
\label{sec:transf-surf-area}

Let $S_{n-1}(L,\cdot)$ be the surface area measure of a symmetric
convex body $L$, see \cite[Sec.~4.2]{schn2}. The $K$-transform of
$S_{n-1}(L,\cdot)$ is the support function given by
\begin{displaymath}
  \int_{\Sphere} h(uK,v)S_{n-1}(L,dv) = nV(L,\dots,L,uK), \quad u\in\R^n,
\end{displaymath}
where $V(L,\dots,L,uK)$ is the mixed volume of $L$ and $uK$. 
The injectivity of $T_K$ implies that the values $V(L,\dots,L,uK)$ for
$u\in\R^n$ uniquely determine the set $L$. Theorems~\ref{thr:main} and
\ref{thr:asym-uniqueness} imply the following fact.

\begin{corollary}
  If $K$ is a generalised zonoid such that all support sets
  $F(K,e_i)$, $i=1,\dots,n$, are singletons, then each unconditional
  convex body $L$ is uniquely determined by the values
  $V(L,\dots,L,uK)$, $u\in\R^n$. A general origin symmetric convex
  body $L$ is uniquely determined if $K$ satisfies the condition of
  Theorem~\ref{thr:asym-uniqueness}.
\end{corollary}

If $K=I$ from \eqref{eq:8}, then $V(L,\dots,L,uK)$ as function of $u$
is the support function of the projection body of $L$, see
\cite[Sec.~5.3.2]{schn2}. For a general $K$, one obtains a
generalisation of the projection body transform, so that
$V(L,\dots,L,uK)$, $u\in\R^n$, is the support function of a convex
body called the $K$-transform of $L$.

\begin{example}
  Let $L=B$ be the unit Euclidean ball. Then $V(B,\dots,B,uK)$ is
  proportional to the mean width of $uK$. The $K$-transform of
  the unit ball has the support function 
  \begin{displaymath}
    \frac{1}{n}\int_{\Sphere} h(uK,v)dv, \quad u\in\R^n. 
  \end{displaymath}
\end{example}

\section{Diagonally transformed $\ell_p$-balls and one-sided stable
  laws}
\label{sec:scaled-ell_p-balls}

\subsection{$D_p$-balls}
\label{sec:d_p-balls}

For a symmetric closed convex set $L$, %(more generally, star-shaped) $L$, 
the \emph{Minkowski functional} is defined by
\begin{displaymath}
  \|u\|_L=\inf\{t>0:\; u\in tL\}, \quad u\in\R^n.
\end{displaymath}
If $L$ is a symmetric convex body with non-empty interior, then $\|u\|_L$ is a norm
with the unit ball being $L$.
Let $\|x\|_p$ be the $p$-norm of $x$ with $p\in[1,\infty]$, and let
\mbox{$B_p=\{x:\; \|x\|_p\leq1\}$} be the $\ell_p$-ball in $\R^n$. 

\begin{definition}
  \label{def:lp1}
  Let $\mu$ be a finite measure on $\Sphere$, and let
  $p\in[1,\infty]$. The convex body $L$ with the
  Minkowski functional
  \begin{equation}
    \label{eq:28a}
    \|u\|_L =\int_{\Sphere} \|uv\|_p\, \mu(dv),\quad u\in\R^n,
  \end{equation}
  is called a \emph{$D_p$-ball}. The measure $\mu$ is called the \emph{spectral measure} of $L$.
\end{definition}

In particular, $B_p$ is a $D_p$-ball with the spectral measure being
the Dirac measure at $(1,\dots,1)$.  Obviously, $D_p$-balls are
unconditional for each $p$ and $\mu$. The measure $\mu$ in
Definition~\ref{def:lp1} can always be chosen either unconditional or
supported by $\Splus=\Sphere\cap\R_+^n$. Since $\mu$ is finite, all
$D_p$-balls have non-empty interior. Moreover $L$ is bounded if and
only if $\mu$ is not supported by a coordinate hyperplane.

It is easy to identify $\|u\|_L$ from \eqref{eq:28a} with $(T_K\mu)(u)$
for $K=B_q$ with $1/p+1/q=1$.  Furthermore, \eqref{eq:28a} can be
expressed as 
\begin{displaymath}
  \|u\|_L =\E\|u\eta\|_p
\end{displaymath}
using an integrable random vector $\eta$.  The polar body $L^\circ$ of the
$D_p$-ball $L$ is a diagonal body generated by~$B_q$ since
\begin{displaymath}
  h(L^\circ,u)=\E \|u\eta\|_p =\E h(uB_q,\eta)=\E h(\eta B_q,u),
  \quad u\in\R^n.
\end{displaymath}
For instance, the polar body of a $D_2$-ball is derived from diagonal
transformations of the Euclidean ball,
see Example~\ref{ex:2}.  
Note that $L$ is unbounded if and only if $L^\circ$ lies in a
coordinate hyperplane.

For $u\in\R_+^n$ and $x\in\R$ we use the notation $u^x=(u_1^x,\dots,u_n^x)$.

\begin{example}[$D_1$-balls]
  If $p=1$, then
  \begin{displaymath}
    \int_{\Sphere} \|uv\|_1\mu(dv)
    =\Big\langle \int_{\Sphere} |v|\mu(dv),|u|\Big\rangle
    =\|uw\|_1=\|u\|_{w^{-1}B_1},
  \end{displaymath}
  whence each $D_1$-ball can be obtained as $w^{-1} B_1$ for some
  $w\in\R^n_+$. A component $w_j$ is zero if and only if $\mu$ is
  supported by the coordinate hyperplane perpendicular to $e_j$. In
  this case the $D_1$-ball is unbounded in the $e_j$-direction.
\end{example}

\begin{remark}
  Definition~\ref{def:lp1} can be extended to $p\in(0,1)$, or to
  averages of arbitrary diagonally transformed norms on $\R^n$; this
  results in a (not necessarily convex) star-shaped set
  $L$. Furthermore, it is possible to introduce an $\Lp[r]$-variant of
  $D_p$-ball with $r\in[1,\infty)$ by considering $(\E
  \|u\eta\|_p^r)^{1/r}$, that is, the $\Lp[r]$-norm of $\|u\eta\|_p$.
  We may also allow $\mu$ to be a signed measure, however restricted
  by the requirement that the right-hand side of \eqref{eq:28a} is
  non-negative for all $u$.
\end{remark}

\subsection{One-sided strictly stable random vectors}
\label{sec:one-sided-strictly}

In the following we show that $D_p$-balls naturally appear in
relation to one-sided stable laws if $p\in[1,\infty)$ and to
max-stable laws if $p=\infty$.  A random vector $\xi$ in $\R^n$ is
\emph{strictly $\alpha$-stable} if
\begin{equation}
  \label{eq:25}
  (t+s)^{1/\alpha}\xi \deq t^{1/\alpha}\xi'+s^{1/\alpha}\xi''
\end{equation}
for all $t,s>0$, where $\xi'$ and $\xi''$ are independent copies of
$\xi$ and $\deq$ denotes the equality in distribution. The parameter $\alpha$ is called \emph{characteristic exponent}. The geometric
interpretation of symmetric stable random vectors is worked out in
\cite{mo09}. In this case, $\alpha\in(0,2]$ (where $\alpha=2$ means
that $\xi$ is Gaussian), and the characteristic function of $\xi$ can
be written as
\begin{equation}
  \label{repr-stable}
  \E \RMe^{\imath \langle \xi,u\rangle} =\exp\{-\|u\|_F^\alpha\}, \quad u\in\R^n,
\end{equation}
where $F$ is an $\LL_\alpha$-ball, that is, 
\begin{displaymath}
  \|u\|_F^\alpha=\int_{\Sphere} |\langle u,v\rangle|^\alpha\sigma(dv)
\end{displaymath}
for a finite measure $\sigma$ on the unit sphere.

If $\xi$ is strictly $\alpha$-stable and $\xi\in\R_+^n$ a.s., then
$\xi$ is said to be \emph{one-sided} or totally skewed to the
right. In this case $\alpha\in(0,1]$, $\alpha=1$ identifies a
deterministic $\xi$, and the Laplace transform of $\xi$ is given by
\begin{equation}
  \label{eq:22}
  \E \RMe^{-\langle \xi,u\rangle} =\exp\Big\{-\int_{\Splus} \langle
  u,v\rangle^\alpha\sigma(dv)\Big\}, \quad u\in\R_+^n,
\end{equation}
where the spectral measure $\sigma$ is finite and supported by
$\Splus$. This follows, e.g., by applying general
results of \cite{dav:mol:zuy08} to the semigroup
$\R_+^n$ with the usual addition and identical involution. 

\begin{theorem}
  \label{thr:one-sided}
  A non-trivial random vector $\xi$ in $\R_+^n$ is strictly
  $\alpha$-stable with $\alpha\in(0,1]$ if and only if there exists a
  $D_{1/\alpha}$-ball $L$ such that the Laplace transform of $\xi$ is
  given by
  \begin{equation}
    \label{eq:23}
    \E \RMe^{-\langle \xi,u\rangle} =\exp\{-\|u^\alpha\|_L\}, \quad u\in\R_+^n.
  \end{equation}
\end{theorem}
\begin{proof}
Sufficiency is immediate, since the Laplace transform of $t^{1/\alpha}\xi$ is $\exp\{-t\|u^\alpha\|_L\}$, whence \eqref{eq:25} holds. To prove necessity, if $\alpha=1$, then $\xi=x$ is deterministic, and $L=x^{-1}B_1$ is a
  $D_1$-ball. Now assume that $\alpha\in(0,1)$ and $\xi$ satisfies \eqref{eq:22}.
  Let $\tilde\eta$ be distributed on $\Splus$ according to the
  normalised $\sigma$, and let
  $\eta=\sigma(\Splus)^{1/\alpha}\tilde\eta$, so that
  \begin{displaymath}
    \int_{\Splus} \langle
    u,v\rangle^\alpha\sigma(dv)=\E(\langle u,\eta\rangle)^\alpha
    =\E \|u^\alpha \eta^\alpha\|_{1/\alpha}= \|u^\alpha\|_L, \quad u\in\R_+^n,
  \end{displaymath}
  where $L$ is the $D_{1/\alpha}$-ball generated by the random vector
  $\eta^\alpha$. 
\end{proof}

\begin{remark}
  \label{rem:max}
  If the arithmetic sum operation on the right-hand side of
  \eqref{eq:25} is replaced by the coordinatewise maximum, $\xi$ is
  said to be \emph{max-stable}. In this case, $\alpha$ can take an
  arbitrary positive value. It is shown in \cite{mo08e} that the
  cumulative distribution functions of max-stable random vectors with
  $\alpha=1$ are characterised as
  \begin{displaymath}
    \P\{\xi\leq u\}=\exp\{-\|u^{-1}\|_L\},\quad u\in(0,\infty)^n,
  \end{displaymath}
  for a $D_\infty$-ball $L$.
  The polar set $L^\circ$ is a
  diagonal body generated by the $\ell_1$-ball in $\R^n$; such convex
  bodies $L^\circ$ were called \emph{max-zonoids} in~\cite{mo08e}.
\end{remark}

Theorem~\ref{thr:one-sided} is supplemented by the following existence result.

\begin{lemma}
For each $D_{1/\alpha}$-ball $L$ with $\alpha\in(0,1]$ there is a strictly $\alpha$-stable random vector $\xi$ in $\R_+^n$ such that the Laplace transform of $\xi$ is given by~\eqref{eq:23}.
\end{lemma}

\begin{proof}
We may assume $\alpha \in (0,1)$, the case $\alpha = 1$ being obvious. Let $\hat{\eta}$ be a random vector in $\R_+^n$ such that $\E \|u \hat{\eta}\|_{1/\alpha}= \|u\|_L$ for $u\in\R_+^n$. Define  $\eta = \hat{\eta}^{1/\alpha}$. Then
\begin{displaymath}
 \|u^\alpha\|_L = \E \|u^\alpha \eta^\alpha\|_{1/\alpha}
 = \E(\langle u,\eta\rangle)^\alpha,
\quad u\in\R_+^n.
\end{displaymath}
Define a finite measure $\sigma(A)=\E(\|\eta\|^\alpha \, \one_{\eta/\|\eta\|\in A})$ for all Borel $A\subset\Splus$. Then  
\begin{equation}
\label{eq:laplace}
    \int_{\Splus} \langle
    u,v\rangle^\alpha\sigma(dv)=\E(\langle u,\eta\rangle)^\alpha, \quad u\in\R_+^n.
\end{equation}

Thus it is enough to show that for each $\alpha \in (0,1)$ and each
finite measure $\sigma$ there is a one-sided strictly $\alpha$-stable
random vector $\xi$ such that \eqref{eq:22} holds.  Choose a sequence
of measures $\sigma_m = \sum_{j=1}^{k_m} a_{mj} \delta_{v_{mj}}$ with
$a_{mj} > 0$, $v_{mj} \in \Splus$, $j = 1,\ldots,k_m$, and $k_m \to
\infty$, such that $\sigma_m \to \sigma$ weakly as $m \to \infty$. For
each $m$ and $j$ define an independent $\alpha$-stable random vector
$\xi_{mj}$ in $\R_+^n$ with completely dependend components having
Laplace transform $\exp (-a_{mj} \langle u, v_{mj}
\rangle^\alpha)$. Clearly the Laplace transform of the sum $\chi_m =
\sum_{j=1}^{k_m} \xi_{mj}$ is
\begin{displaymath}
  \E e^{-\langle u,\chi_m\rangle}=
  \exp \left\{-\int_{\Splus} \langle u,v \rangle^\alpha \sigma_m(dv)\right\}, \quad u\in\R_+^n.
\end{displaymath}
By the pointwise convergence of Laplace transforms, the limit in
distribution of $\chi_m$ exists and has the Laplace
transform~\eqref{eq:laplace}, see for example \cite[Th.~5.22]{kalle}.
From its Laplace transform it is also obvious that the limit of
$\chi_m$ is strictly $\alpha$-stable.
\end{proof}

\begin{example}[Completely dependent components]
  \label{ex:sp-c-dep}
  If the spectral measure is the Dirac measure at $v\in\R_+^n$, the
  corresponding $D_p$-ball is $v^{-1}B_p$ for $p\in[1,\infty)$.  The
  corresponding one-sided $\alpha$-stable random vector with
  $\alpha=1/p$ satisfies
  \begin{displaymath}
    \E \RMe^{-\langle\xi,u\rangle}=\RMe^{-\|vu^\alpha\|_p}
    =\exp\Big\{-(u_1v_1^{1/\alpha}+\cdots
    +u_nv_n^{1/\alpha})^{\alpha}\Big\},
    \quad u\in\R_+^n.
  \end{displaymath}
  This is the Laplace transform of the random vector $\zeta
  v^{1/\alpha}$ obtained by scaling all components of $v^{1/\alpha}$
  with a one-sided strictly $\alpha$-stable random variable $\zeta$.
  In other words, the components of $\xi$ are completely dependent.
\end{example}  

\begin{example}[Independent components]
  \label{ex:sp-c-indep}
  If the components of one-sided strictly stable $\xi$ are
  independent, then
  \begin{displaymath}
    \E \RMe^{-\langle\xi,u\rangle}
    =\exp\Big\{-(v_1 u_1^\alpha+\cdots+v_n u_n^\alpha)\Big\}
    =\exp\big\{-\|u^\alpha\|_{v^{-1} B_1}\big\},
    \quad u\in\R_+^n,
  \end{displaymath}
  for some $v\in\R^n_+$. It follows that $B_1$ and thus any $D_1$-ball 
  is a $D_p$-ball for each $p\in(1,\infty)$. 
\end{example}

For a closed convex set $L$ and $\beta>0$, define its signed $\beta$-power by 
\begin{displaymath}
  L^{\langle \beta\rangle}=\{x^{\langle \beta\rangle}:\; x\in L\},
\end{displaymath}
where $x^{\langle \beta\rangle}$ is the vector composed of the signed
powers of the components of $x$. 

\begin{lemma}
  \label{lemm:sc}
  If $L$ is a $D_p$-ball with $p\in[1,\infty)$ and $\beta\in(0,1)$,
  then $L^{\langle \beta\rangle}$ is a $D_{p/\beta}$-ball.
\end{lemma}

\begin{proof}
  Let $\xi$ be one-sided stable in $\R_+^n$ with characteristic exponent $\alpha = 1/p$ such that \eqref{eq:23} holds. Further let $\zeta$ be a one-sided stable random variable with
  characteristic exponent $\beta\in(0,1)$ and Laplace transform
  $\exp(-s^\beta)$, $s \geq 0$. Then $\zeta^{1/\alpha}\xi$ is also
  one-sided stable (also called sub-stable, see \cite{sam:taq94}) with
  characteristic exponent $\alpha\beta$, and
  \begin{displaymath}
    \E \RMe^{-\langle \zeta^{1/\alpha}\xi,u\rangle} = \exp\{-\|u^\alpha\|_L^\beta\}
    =\exp\{-\|u^{\alpha\beta}\|_{L'}\}\,,
  \end{displaymath}
  where $L'$ correspond to $\zeta^{1/\alpha}\xi$, i.e.\ $L'$ is a $D_{p'}$-ball with $p'=1/(\alpha\beta)=p/\beta$.
  Then
  \begin{align*}
    \|u\|_{L'} = \|u^{1/\beta}\|_L^\beta
    =\inf\{s> 0:\; u^{1/\beta}/s\in L\}^\beta
    &=\inf\{t> 0:\; (u/t)^{1/\beta}\in L\}\\
    &=\|u\|_{L^{\langle\beta\rangle}}. 
 \end{align*}
 It follows that $L'=L^{\langle \beta\rangle}$.
\end{proof}

\begin{example}
  \label{ex:lpball}
  If $B_p$ is the $\ell_p$-ball and so is a $D_p$-ball, then
  $B_p^{\langle p/r\rangle}=B_r$. For $r>p$, this corresponds to the
  conclusion of Lemma~\ref{lemm:sc}.
\end{example}

\begin{lemma}
  \label{lemma:dimension}
  Let $L$ be a symmetric convex body, let $\zeta$ be an integrable
  random vector in $\R^n$, and define $K=\E(\zeta L)$. For $j \in
  \{1,\ldots,n\}$, if the support set $F(K,e_j)$ is a singleton, then
  $\P(\zeta_j=0)=0$. The converse implication holds if $F(L,e_j)$ is a
  singleton.
\end{lemma}
\begin{proof}
  The first statement is clear. Now assume that $F(L,e_j) = \{ y \}$
  for some $j \in \{1,\ldots,n\}$ and $y \in L$ and that
  $\P(\zeta_j=0)=0$. Then the derivative of $h(L,u)$ at $e_j$ in
  direction $x \in \R^n$ is $h_L'(e_j ; x) = h(F(L,e_j),x) = \langle
  y, x \rangle$, see \cite[Th.~1.7.2]{schn2}. Then
  \begin{align*}
    h_K'(e_j ; x)
    % & =  \lim_{\varepsilon \downarrow 0} h(K, e_j + \varepsilon x)\\
    & =  \lim_{\varepsilon \downarrow 0} \E h(\zeta L, e_j + \varepsilon x)
    =  \lim_{\varepsilon \downarrow 0} \E h(L, \zeta_j e_j + \varepsilon \zeta x)\\
    & =  \lim_{\varepsilon \downarrow 0} \E \left(|\zeta_j| \, 
      h \left(L, e_j + \varepsilon \zeta_j^{-1} \zeta x \right)\right)\\
    & =  \E \left(|\zeta_j| \, h_L' \left(e_j ; \zeta_j^{-1} \zeta x \right)\right)
    =  \langle y \, \E \left(\zeta \sign(\zeta_j)\right), x \rangle,
  \end{align*}
  confirming the linearity of $h_K'(e_j ; x)$ in $x$, and so the
  second claim by \cite[Cor.~1.7.3]{schn2}. 
\end{proof}

\begin{theorem}
  \label{th:drball}
  If $L$ is a $D_p$-ball for some $p \in [1,\infty)$, then $L$ is a
  $D_r$-ball for $r \in (p,\infty]$.
\end{theorem}
\begin{proof}
  First consider $r < \infty$. It suffices to prove the claim for
  $L=B_p$. The case $p = 1$ is shown in
  Example~\ref{ex:sp-c-indep}. If $p > 1$, then $B_p=B_1^{\langle
    1/p\rangle}$. Since $B_1$ is a $D_s$-ball for $s>1$, $B_p$ is a
  $D_{sp}$-ball by Lemma~\ref{lemm:sc}.

  Now consider $r = \infty$. Since $L$ is a $D_s$-ball for all
  $s\in[p,\infty)$, for each $s \in [p,\infty)$ there is a finite
  measure $\mu_s$ on $\SS_+^{n-1}$ such that
  \begin{displaymath}
    \|u\|_L =\int_{\SS_+^{n-1}} \|uv\|_s\, \mu_s(dv),\quad u\in\R^n.
  \end{displaymath}
  We have $\|v\|_s \geq \|v\|_\infty > c$ for $v \in \SS_+^{n-1}$ and
  some $c>0$. Thus
  \begin{displaymath}
    \|(1,\ldots,1)\|_L = \int_{\SS_+^{n-1}} \|v\|_s\, \mu_s(dv)
    \geq c \mu_s(\SS_+^{n-1}).
  \end{displaymath}
  This shows that $\{ \mu_s(\SS_+^{n-1}) : s \in [p,\infty)\}$ is
  bounded. Now choose a sequence of numbers $r_k \to \infty$, $k \geq
  1$, and random vectors $\xi^{(k)} = \mu_{r_k}(\SS_+^{n-1})
  \eta^{(k)}$ where $\eta^{(k)}$ is distributed as
  normalised~$\mu_{r_k}$. Then $\{\xi^{(k)}, k \geq 1\}$ is a tight
  sequence, and
  \begin{displaymath}
    \|u\|_L = \E \|u\xi^{(k)}\|_{r_k},\quad u\in\R^n.
  \end{displaymath}
  Without loss of generality assume that $\xi^{(k)} \to \xi$ almost
  surely as $k \to \infty$ for some random vector~$\xi$. By bounded
  convergence $\E \|u\xi^{(k)}\|_{r_k} \to \E \|u\xi\|_\infty$ for
  $u\in\R^n$. Thus, $L$ is a $D_\infty$-ball.
\end{proof}

\begin{corollary}
  \label{co:ballzon}
  If $L$ is a $D_p$-ball for some $p\in[1,2]$, then $L^\circ$ is a zonoid.
\end{corollary}
\begin{proof}
  Since the Euclidean ball $B_2$ is a zonoid, the polar set of each
  $D_2$-ball is a zonoid. For $p\in[1,2)$ use Theorem~\ref{th:drball}.
\end{proof}

\subsection{Unconditional D-universality of $\ell_p$-balls}
\label{sec:uncond-d-univ}

Given that $\|u\xi\|_p=h(\xi B_q,u)$, the following results show that
$\ell_q$-balls are unconditionally D-universal for all
$q\in[1,\infty)$ that is, $\E \|u\xi\|_p$, $u\in\R^n$, identify the
distribution of $|\xi|$ up to zonoid equivalence. Recall that this is
not true for $q=\infty$, see Example~\ref{ex:1}. For $q\in[2,\infty)$,
$B_q$ is a zonoid whose support sets satisfy the conditions of
Theorem~\ref{thr:main}, whence $B_q$ is unconditionally
D-universal. However, $B_q$ is not necessarily a zonoid if
$n\geq 3$, $q\in[1,2)$, e.g.\ if $q$ is close to one. 

\begin{theorem}
  \label{th:maxzon}
  Let $K$ be a symmetric convex body such that all support sets
  $F(K,e_i)$, $i=1,\dots,n$, are singletons, and such that $K^\circ$
  is a $D_p$-ball for some $p\in(1,\infty]$. Then $K$ is
  unconditionally D-universal.
\end{theorem}
\begin{proof}
  By Theorem~\ref{th:drball} we may assume that $p=\infty$. Now assume
  that \eqref{eq:12} holds for $K=\E(\zeta B_1)$ where $\zeta$ is a
  random vector in $\R_+^n$ independent of $\xi$ and $\eta$. This
  implies
  \begin{displaymath}
    \E\|u\zeta\xi\|_\infty = \E\|u\zeta\eta\|_\infty,\quad u\in\R^n.
  \end{displaymath}
  By Theorem~\ref{thr:ws},
  \begin{displaymath}
    \E|\langle u, \zeta |\xi| \rangle| 
    = \E|\langle u,\zeta |\eta| \rangle|,\quad u\in\R^n.
  \end{displaymath}
  Now let $L=\E[-\zeta,\zeta]$ so that
  \begin{displaymath}
    \E h(uL,|\xi|)=\E h(uL,|\eta|),\quad u\in\R^n.
  \end{displaymath}
  Since the support sets of the zonoid $L$ in coordinate directions
  are singletons by Lemma~\ref{lemma:dimension},
  the claim follows from Theorem~\ref{thr:main}.
\end{proof}

\begin{corollary}
  \label{th:lpzonoid}
  If $\xi$ and $\eta$ are integrable random vectors in
  $\R_+^n$ and $p\in(1,\infty]$, then
  \begin{equation}
  \label{eq:pball}
    \E \|u\xi\|_p=\E\|u\eta\|_p,\quad u\in\R^n, 
  \end{equation}
  if and only if $\xi$ and $\eta$ are zonoid equivalent.
\end{corollary}
\begin{proof}
  In Theorem~\ref{th:maxzon} set $K = B_q$ with $1/p+1/q=1$.
\end{proof}

In the case of strictly positive $\xi$ and $\eta$, an analytical proof
of the equivalence in Corollary~\ref{th:lpzonoid} is given in
\cite{mol:sch11m}. Note that the case $p=1$ is excluded (see
Example~\ref{ex:1}); it would correspond to $K=B_\infty$ which does
not satisfy the condition on support sets imposed in
Theorem~\ref{th:maxzon}.

By Theorems~\ref{thr:equiv-pos} and \ref{th:maxzon}, linear
combinations of support functions of polar sets of $D_p$-balls are
dense in the family of support functions of unconditional convex
bodies if $p\in(1,\infty]$ and if the conditions on the support sets are satisfied.

\renewcommand{\thesection}{Appendix:}
\renewcommand{\thetheorem}{A.\arabic{theorem}}
\renewcommand{\theequation}{A.\arabic{equation}}

\section{General transformations}
\label{sec:appendix}

For completeness, we mention the following result generalising
Theorem~\ref{thr:equiv} to a more general subfamily $G$ of linear
transformations on $\R^n$. The $G$-invariant Minkowski class generated
by $K\in\sK_0$ consists of the limits in the Hausdorff
metric of the linear combinations
\begin{displaymath}
  \alpha_1 g_1K+\cdots+\alpha_m g_m K, 
\end{displaymath}
where $\alpha_1,\dots,\alpha_m>0$, $g_1,\dots,g_m\in G$ and $m\geq1$.
Due to the presence of scaling constants $\alpha_i$, it is possible to
assume that $G$ is bounded.  Let $\nu$ be a finite signed measure on $G$
equipped with the Borel $\sigma$-algebra. Denote
\begin{displaymath}
  (T_{K,G}\nu)(u)=\int_G h(gK,u)\nu(dg),\quad u\in\R^n.
\end{displaymath}
The dominated convergence theorem yields that $T_{K,G}\nu$ is
continuous.  Unlike the case of diagonal transformations, it is not
possible to swap $g$ and $u$ in $h(gK,u)$.  The following result can
be derived following the same arguments as in \cite[Th.~3.5.3]{schn2}.

\begin{proposition}
  \label{prop:m-class}
  A convex body $L$ belongs to the $G$-invariant Minkowski class
  generated by $K$ if and only if $h(L,u)=(T_{K,G}\nu)(u)$ for a
  finite even measure $\nu$ on $G$.
\end{proposition}

\begin{theorem}
  \label{thr:group}
  The linear space spanned by the support functions from the
  $G$-invariant Minkowski class generated by $K$ is dense in the
  family of support functions of convex bodies from $\sK_0$ if and
  only if the transform
  \begin{equation}
    \label{eq:15}
    \mu\mapsto \int_{\Sphere} h(gK,v)\mu(dv), \quad g\in G,
  \end{equation}
  is injective for finite even signed measures $\mu$ on $\Sphere$.
\end{theorem}
\begin{proof}
  \textsl{Necessity.}  The support function $|\langle u,v\rangle|$ of
  the segment $uI$ can be approximated by linear combinations of the
  support functions $h(g_iK,v)$. Hence, if the right-hand side of
  \eqref{eq:15} vanishes for all $g\in G$, the cosine transform of
  $\mu$ vanishes. So the transform \eqref{eq:15} is injective.
 
  \textsl{Sufficiency.}  
  Denote by $\sM_G$ the family of finite signed measures on $G$ with
  the total variation norm. The operator $T_{K,G}$ is continuous 
  and maps measures from $\sM_G$ to continuous even functions
  on the unit sphere.  Its adjoint $T'_{K,G}$ is an operator on the
  family $\sM_e$ of signed finite even measures on $\Sphere$, and
  $T'_{K,G}\mu$ belongs to the dual space of $\sM_G$. For $\mu \in
  \sM_e$ and $\nu\in\sM_G$,
  \begin{displaymath}
    \langle T'_{K,G} \mu,\nu\rangle = \langle \mu,T_{K,G}\nu\rangle
    =\int_{\Sphere} (T_{K,G} \nu)d\mu,
  \end{displaymath}
  where the left-hand side refers to the pairing of $T'_{K,G} \mu$ and
  $\nu$.  If $T'_{K,G} \mu = 0$, then
  \begin{displaymath}
    \int_{\Sphere}\int_G h(gK,u)\nu(dg)\mu(du)=0
  \end{displaymath}
  for all $\nu\in\sM_G$. Changing the order of integration yields that 
  \begin{displaymath}
    \int_{\Sphere} h(gK,u)\mu(du)=0,
  \end{displaymath}
  whence $\mu = 0$ by injectivity. The triviality of the kernel of
  $T'_{K,G}$ yields that the range of $T_{K,G}$ is dense in the space of
  continuous functions on $\Sphere$, see \cite[Th.~III.4.5]{wern00}.
\end{proof}

If $G$ is the group of all invertible linear transformations, then
\cite[Th.~5(1)]{ales03} yields that, for each symmetric convex body
$K$, the equality $\E h(gK,\xi)=\E h(gK,\eta)$, $g\in G$, holds if and
only if $\xi$ and $\eta$ are zonoid equivalent. 

Similar results hold in the unconditional case meaning that the linear
combinations of support functions of convex bodies from the
$G$-invariant Minkowski class are dense in the family of support
functions of convex bodies from $\sK_s$ if and only if $\E
h(gK,\xi)=\E h(gK,\eta)$ for all $g\in G$ yields that $|\xi|$ and
$|\eta|$ are zonoid equivalent.  Note that, for general $G$, the
Minkowski class may contain convex bodies that are not in $\sK_s$
though $K$ is from $\sK_s$.  In the special case of diagonal
transformations this does not occur.

\section*{Acknowledgement}
\label{sec:acknowledgement}

This work was supported by the Swiss National Science Foundation Grant
200021\_153597.

%\bibliographystyle{plain}
%\bibliography{/Users/ilya/tex/bibdata/abbrev,/Users/ilya/tex/bibdata/x/gesamt}

\newcommand{\noopsort}[1]{} \newcommand{\printfirst}[2]{#1}
  \newcommand{\singleletter}[1]{#1} \newcommand{\switchargs}[2]{#2#1}

\end{document}